\documentclass[10pt]{amsart}

\usepackage[all]{xy}
\usepackage[latin1]{inputenc} 
\usepackage[T1]{fontenc}
\usepackage[dvips]{graphics,graphicx}
\usepackage{amsfonts,mathabx}
\usepackage{amssymb}
\usepackage{amsmath, amsrefs}
\usepackage{mathtools}
\usepackage{esint}
\allowdisplaybreaks

\usepackage{amsbsy,
	amsopn,
	amscd, 
	amsxtra, 
	amsthm,
	verbatim}
\usepackage{upref}
\usepackage{color}
\usepackage{enumitem}
\usepackage[colorlinks,
linkcolor=red,
anchorcolor=red,
citecolor=red
]{hyperref}

\usepackage{thmtools}
\usepackage{caption}
\usepackage{subcaption}

\usepackage{booktabs}
\usepackage{longtable}

\newtheorem{theorem}{Theorem}
\newtheorem{lemma}{Lemma}[section]

\newtheorem*{proposition*}{Proposition}

\newtheorem*{corollary*}{Corollary}

\newtheorem{assumption}{Assumption}

\theoremstyle{remark}
\newtheorem{remark}[lemma]{Remark}

\usepackage{geometry}
\usepackage{float}

\usepackage{todonotes}

\newcommand{\eg}{{e.g.}}
\makeatletter
\newcommand*{\rom}[1]{\expandafter\@slowromancap\romannumeral #1@}
\makeatother

\newcommand{\wb}[1]{\widebar{#1}}
\newcommand{\mc}[1]{\mathcal{#1}}

\providecommand{\keywords}[1]{\textbf{\textit{Keywords: }} #1}

\newcommand{\ud}{\,\mathrm{d}}
\newcommand{\rd}{\mathrm{d}}

\newcommand{\RR}{\mathbb{R}}

\newcommand{\abs}[1]{\lvert#1\rvert}

\newcommand{\norm}[1]{\lVert#1\rVert}
\newcommand{\Lap}{\Delta_{t,x}}

\newlength{\leftstackrelawd}
\newlength{\leftstackrelbwd}
\def\leftstackrel#1#2{\settowidth{\leftstackrelawd}%
{${{}^{#1}}$}\settowidth{\leftstackrelbwd}{$#2$}%
\addtolength{\leftstackrelawd}{-\leftstackrelbwd}%
\leavevmode\ifthenelse{\lengthtest{\leftstackrelawd>0pt}}%
{\kern-.5\leftstackrelawd}{}\mathrel{\mathop{#2}\limits^{#1}}}

\def\bigl{\mathopen\big}

\def\bigr{\mathclose\big}

\newcommand{\id}{\mathcal{I}}

\newcommand{\A}{\mathcal{A}}
\newcommand{\B}{\mathcal{B}}
\newcommand{\J}{\mathcal{J}}
\newcommand{\nn}{\mathfrak{n}}
\newcommand{\opL}{\mathcal{L}}

\newcommand{\stationary}{\rho_{\infty}}

\title[$L^2$-convergence rate for PDMP]{On explicit $L^2$-convergence rate estimate for piecewise deterministic Markov processes in MCMC Algorithms}

\author{Jianfeng Lu}
\address{Department of Mathematics, Department of Physics, and Department of Chemistry, Duke University, Durham NC 27708}
\email{jianfeng@math.duke.edu}
\author{Lihan Wang}
\address{Department of Mathematics, Duke University, Durham NC 27708}
\email{lihan@math.duke.edu} \date{\today}

\begin{document}

\begin{abstract}
We establish $L^2$-exponential convergence rate for three popular piecewise deterministic Markov processes for sampling: the randomized Hamiltonian Monte Carlo method, the zigzag process, and the bouncy particle sampler. Our analysis is based on a variational framework for hypocoercivity, which combines a Poincar\'{e}-type inequality in time-augmented state space and a standard $L^2$ energy estimate. Our analysis provides explicit convergence rate estimates, which are more quantitative than existing results. 
\end{abstract}
\maketitle

\section{Introduction}
Sampling approaches based on piecewise deterministic Markov processes
(PDMPs) \cite{davis1984piecewise}, which involve random jumps and
deterministic trajectories in between, have recently attracted a lot
of attention: several classes of Markov Chain Monte Carlo (MCMC)
algorithms have been developed based on PDMPs, including the
randomized Hamiltonian Monte Carlo (RHMC)
\cite{duane1987hybrid, bou2017randomized}, the zigzag process \cite{bierkens2019zig} and the
bouncy particle sampler \cites{peters2012rejection,
  bouchard2018bouncy}. Compared with MCMC algorithms based on
diffusion, such as overdamped and underdamped Langevin Monte Carlo,
the methods based on PDMPs do not need time discretization for the
random part and the deterministic dynamics can either be explicitly
integrated (for zigzag and bouncy particle) or be dealt with high
order numerical integration (for RHMC), which make them promising to
have better numerical performance \cites{bierkens2016non,
  bierkens2018piecewise, fearnhead2018piecewise, bou2017randomized, bou2018geometric, sen2020efficient}. The
zigzag and bouncy particle samplers are also suitable for the big data
situation, as they can be unbiased even if stochastic gradient is
used \cites{bouchard2018bouncy, bierkens2019zig}. 

Typical PDMPs for sampling purpose introduce an auxiliary ``velocity''
variable $v\in \RR^d$ that facilitates simulation, which is often
chosen from a fixed distribution. For this paper, we will only
consider the case that the velocity variable is drawn from the standard
Gaussian distribution
$\ud \kappa(v)=(2\pi)^{-\frac{d}{2}}e^{-\frac{|v|^2}{2}}\ud v$.  In
the PDMPs, the velocity variable is redrawn independently from the
Gaussian distribution at a certain rate, and between two redraws the
trajectory of state variable $(x, v)$ consists of deterministic routes
and random bounces so that the spatial variable $x$ will explore the
state space in all different directions with the help of $v$. The
PDMPs are designed so that the $x$ samples the desired target
distribution. 

We now present the general mathematical formulation of PDMPs. Let
$f=f(t,x,v): \, \RR_+\times \RR^d \times \RR^d \to \RR $ be the
expectation of some observable function $f_0(x,v)$ at time $t$, and
therefore satisfies the backward Kolmogorov equation
\begin{equation}\label{eqn:pdmpeqn}
  \partial_t f = \opL f, \qquad f(t=0,x,v)=f_0(x,v),
\end{equation}
where the infinitesimal generator $\opL$ associated with PDMPs is
given by
\begin{equation}\label{eqn:opL}
  \opL=v\cdot \nabla_x-F_0(x) \cdot \nabla_v +\sum_{k=1}^K (v\cdot F_k(x))_+ (\B_k-\id) +\gamma (\Pi_v-\id) .
\end{equation}
Here $\id$ denotes the identity operator, and $(s)_+:=\max\{s,0\}$. The vector fields
$F_k: \RR^d \to \RR^d$, $k = 0, 1, \ldots, K$ depend only on the
position variable $x$ (examples will be discussed below). The jump
operators $\B_k$ correspond to reflections of the velocity variable
through the hyperplanes orthogonal to $F_k$, defined as
\begin{equation}\label{eqn:bk}
  \B_k f(t,x,v):=f\bigl(t,x,v-2(v\cdot n_k(x))n_k(x)\bigr), 
\end{equation}
where
\begin{equation}
  n_k(x) = \begin{cases}
    F_k(x)/|F_k(x)| & \mbox{ if }F_k(x)\neq 0, \\
    0 & \mbox{ otherwise},
  \end{cases}
\end{equation}
and $\Pi_v$ is the projection operator on Gaussian with respect to $v$ variable
\begin{equation}
  \label{eqn:Pi_v}
  (\Pi_v f)(t,x) := \int_{\RR^d} f(t,x,v)\ud \kappa(v).
\end{equation}
In \eqref{eqn:opL}, $\gamma>0$ is the refreshment rate of the
velocity variable, whose choice will impact the convergence rate of
the dynamics. Our analysis will provide optimal choices of $\gamma$.

Different PDMPs correspond to different choices of the vector fields
$F_k$. While our framework can be generalized to various situations, for definiteness, we will only focus on the three most prominent
examples:
\begin{itemize}
\item The randomized Hamiltonian Monte Carlo (RHMC) \cite{duane1987hybrid,bou2017randomized} corresponds to the choice $K=0$ and $F_0=\nabla_x U$, where $U$ is some potential function.  The corresponding equation \eqref{eqn:pdmpeqn} can be seen as a particular linear Boltzmann equation \cite{bhatnagar1954model} with the collision operator given by $\gamma (\Pi_v - \mc{I})$;
\item The zigzag process (ZZ) \cite{bierkens2019zig} corresponds to $K=d$, $F_0=0$ and $F_k=\partial_{x_k}U e_k$ where $(e_k)_{k\in \{1,\cdots,d\}}$ is the canonical basis of $\RR^d$; 
\item The bouncy particle sampler (BPS) corresponds to the choice $K=1$, $F_0=0$ and $F_1=\nabla_x U$. The BPS was first proposed in \cite{peters2012rejection} and extended in \cite{bouchard2018bouncy}.
\end{itemize}
All these PDMP processes above satisfy $\sum_{k=0}^K F_k=\nabla_x U$, and thus are designed so that they admit a unique stationary distribution \cites{bou2017randomized, bierkens2019ergodicity, monmarche2016piecewise} given by
\begin{align}
	\ud \stationary(x, v) = \ud \mu_U(x) \ud \kappa(v),
\end{align}
where 
\begin{align*}
	\ud \mu_U(x) = \frac{1}{Z_U}e^{-U(x)}\ud x, \quad \qquad Z_U = \int_{\RR^d} e^{- U(x)}\ud x.
\end{align*}
Other PDMPs have been proposed for sampling purposes, including 
Hamiltonian BPS \cite{vanetti2017piecewise}, the Coordinate Sampler \cite{robert2019coordinate}, the Gibbs zigzag sampler \cite{sachs2020posterior}, the Boomerang sampler \cite{bierkens2020boomerang}, and more general bounces involving randomization \cite{vanetti2017piecewise, michel2020forward,wu2017generalized}. While our framework can be generalized to these algorithms, we will not consider these variants in this work. 
	
Our goal is to derive explicit decay rate estimates in $L^2$ for PDMPs, based on the variational framework developed in \cite{armstrong2019variational} and our previous work for the underdamped Langevin dynamics \cite{cao2019explicit}, the idea of which originates from the pioneering work \cite{hormander1967hypoelliptic}. More precisely, we will obtain explicit estimates for some $\nu > 0$ and a universal constant $C>1$ independent of $U,~\gamma$ and $d$ such that for $f=f(t,x,v)$ solving \eqref{eqn:opL} and $\int f_0 \ud\stationary = 0$, we have
\begin{equation}  \label{eqn::exp_decay}
  \norm{f(t, \cdot, \cdot)}_{L^2_{\stationary}} \le Ce^{-\nu t}\norm{f_0}_{L^2_{\stationary}} .
\end{equation} 

Geometric convergence for ZZ has been established in \cite{bierkens2019ergodicity} and for BPS in \cite{monmarche2016piecewise, deligiannidis2019exponential, durmus2018geometric}, however the expressions of the convergence rates are either implicit or complicated. The work \cite{andrieu2018hypocoercivity} established explicit convergence rate for these processes, however only in terms of the dimension $d$; the comparison of their result with ours will be further elaborated below after we present our main results. 

Other theoretical studies of the PDMPs include scaling limits and spectral analysis: The work \cite{deligiannidis2018randomized} established the scaling limit of first coordinate for BPS, and \cite{bierkens2018high} proved scaling limits of ZZ and BPS for several statistical observables. Spectral analysis of PDMP were considered in \cite{miclo2013etude, bierkens2019spectral} in one-dimension and \cite{guillin2020low} for the metastable regime. 

More generally, convergence result of type \eqref{eqn::exp_decay} for hypocoercive equations was established in $H^1(\stationary)$ in \cite{villani_hypocoercivity_2009, mouhot2006quantitative} for a class of kinetic equations. Hypocoercivity estimate in terms of a modified $L^2$ space was developed in \cites{herau2006hypocoercivity, dolbeault2009hypocoercivity, dolbeault2015hypocoercivity} and a series of works based on this framework \cite{roussel2018spectral,bernard2020hypocoercivity,grothaus2014hypocoercivity}. 

\subsection*{Notations}
Throughout the paper we assume $I$ to be the
time interval $(0,T)$, and we use $\ud \lambda(t)=\chi_{(0,T)}(t)\ud t$ to denote the Lebesgue measure on $I$. Define $C_b^k$ to be the set of functions $f$ such that are $k$-times differentiable with bounded derivatives up to order $k$. We define the Sobolev space
\[H^1(\mu_U):=\{f: f(x)\in L^2(\mu_U) \mbox{ and
  }\partial_{x_k} f\in L^2(\mu_U), \ \forall k=1,\cdots,d
  \}.\]
We also define
\[L^2(\lambda\times \mu_U):=\{f=f(t,x): \int_{I\times\RR^d} f^2\ud
  t\ud\mu_U(x)<\infty\},\]and its corresponding
norm
\begin{equation*}
  \norm{f}_{L^2(\lambda\times \mu_U)} := \bigl(\int_{I\times\RR^d}
  f^2\ud t\ud\mu_U(x)\bigr)^\frac{1}{2}.
\end{equation*}
The space $L^2(\lambda\times \stationary)$ for functions on
$I \times \RR^d \times \RR^d$ and its corresponding norm are defined
similarly. We define the average of $f: I \times \RR^d \times \RR^d \to \RR$ over $\lambda\times \stationary$ as
\begin{equation*}
  (f)_{\lambda\times \stationary} :=\dfrac{1}{T}\int_{I\times \mathbb{R}^d\times
    \mathbb{R}^d} f(t,x,v)\ud t\ud\stationary(x,v),
\end{equation*} and for $g: I\times \RR^d\to \RR$ we define its average over $\lambda\times \mu_U$ as \begin{equation*}
  (g)_{\lambda\times \mu_U} :=\dfrac{1}{T}\int_{I\times \mathbb{R}^d} g(t,x)\ud t\ud\mu_U(x).
\end{equation*}
We use
\[\nabla_x^*F:=-\nabla_x\cdot F+F\cdot\nabla_x U\] to denote the
$L^2(\mu_U)$-adjoint operator of $\nabla_x$. For time-augmented state space $I\times \RR^d$ equipped with measure $\lambda\times \mu_U$, we use the convention $\partial_{x_0}:=\partial_t$, the short-hand notation $\wb{\nabla}:=(\partial_t,\nabla_x)^\top$, and the notation $\Lap:=-\partial_{tt}+\nabla_x^*\nabla_x$ to denote the ``Laplace'' operator on $L^2(\lambda\times \mu_U)$.

\subsection{Assumptions and Main Results}Below are three fundamental assumptions that $U(x)$ must satisfy in our framework. The convergence rate gets better if we have stronger assumptions on $U$.

\begin{assumption}[Poincar\'{e} inequality for $\mu_U$]\label{assump:poincare}
  The measure $\ud\mu_U$ corresponding to $U(x)$ satisfies a Poincar\'{e}
  inequality with constant $m>0$:
  \begin{equation}
    \label{eqn:spatialpoincare}
    \int_{\mathbb{R}^d} \left(f-\int_{\mathbb{R}^d} f \ud\mu_U\right)^2\rd\mu_U \le \dfrac{1}{m}\int_{\mathbb{R}^d} |\nabla f|^2 \ud\mu_U, \qquad \forall f\in H^1(\mu_U).
  \end{equation}
\end{assumption}

\begin{assumption}\label{assump:hessian}
  The potential $U\in C^2(\RR^d)$, and the Hessian of $U$,
  $\nabla^2 U$ satisfies
  \begin{equation}\label{eqn:stoltzcond9}
    \|\nabla^2 U(x) \| \le M(1 + |\nabla U(x)|), \qquad \forall \ x\in \mathbb{R}^d
  \end{equation}
  for some constant $M\ge 1$, where $\norm{\cdot}$ denotes the matrix
  operator norm
  \[\norm{A}:=\sup_{\xi\in \RR^d \backslash \{0\}} \dfrac{|A
      \xi|}{|\xi|}.\]
\end{assumption}

\begin{assumption}\label{assump:spectral}
  The embedding $H^1(\mu_U) \xhookrightarrow{} L^2(\mu_U)$ is compact.
\end{assumption}

The Assumption~\ref{assump:hessian} is commonly used in the
literature, see \eg{}, the books
\cite{pavliotis2014stochastic,villani_hypocoercivity_2009} for
underdamped Langevin dynamics, and is satisfied when $U$ grows at most
exponentially fast as $x\to\infty$. Assumption~\ref{assump:spectral}
is satisfied as long as
\begin{equation*}
  \lim_{|x|\to\infty}\dfrac{U(x)}{|x|^\alpha}=\infty
\end{equation*}
for some $\alpha>1$ (see \cite{hooton1981compact} for a proof). While
previous works on hypocoercivity \cite{dolbeault2015hypocoercivity}
and works following its framework \cite{roussel2018spectral,
  andrieu2018hypocoercivity, bernard2020hypocoercivity} use elliptic
regularity estimate in $x$ for which Assumption \ref{assump:poincare}
suffices, our proof, in particular the construction of test functions
in Lemma \ref{lem:truetestfn}, relies on spectral decomposition of the
operator $\nabla_x^*\nabla_x$, which is only guaranteed through the
slightly stronger Assumption \ref{assump:spectral}.

\smallskip

It is established in \cite{durmus2018piecewise} that BPS and ZZ are well-defined Markov process whose generators admit $C_b^2(\RR^d\times \RR^d)$ as a core, and similar arguments can be used to establish that RHMC generator has the same core. It is standard that $C_b^2(\RR^d\times \RR^d)$ is dense in $L^2(\stationary)$. Moreover, the operator $\opL$ is closed in $L^2(\stationary)$, and generates a strongly continuous contraction semigroup $(P_t)_{t\ge 0}$ on $L^2(\stationary)$. This sets up the basic regularity assumptions needed in this work.

\smallskip
	
Below we present the main result of this work.
\begin{theorem}\label{thm:expconvfull}
  Under Assumptions~\ref{assump:poincare}, ~\ref{assump:hessian},
  and~\ref{assump:spectral}, there exist a constant $\nu > 0$ and
  universal constants $C_0,c_0$ independent of all parameters such
  that, for any $f$ satisfying $f_0 \in L^2(\stationary)$ and
  \begin{equation}\label{eqn:meanzero}
    \int_{\mathbb{R}^d\times \mathbb{R}^d} f_0\ud\stationary(x,v)=0,
  \end{equation} 
  and solving the PDMP equation \eqref{eqn:pdmpeqn}, we have for every $t>0$,
  \begin{equation}\label{eqn:expcvg}
    \norm{f(t,\cdot)}_{L^2(\stationary)} \le C_0\exp(-\nu t)\|f_0\|_{L^2(\stationary)}.
  \end{equation}
  Moreover, let $R$ be the parameter that describes the ``convexity
  barrier'' of $U$, defined as
  \begin{equation}\label{eqn:convbar}
    R=R(U):= \begin{cases}
      0, & \text{if } U \text{ is convex}; \\
      \sqrt{L},  & \text{if } \nabla_x^2 U(x) \ge -L\id, \ \forall x; \\
      M\sqrt{d}, & \text{if only \eqref{eqn:stoltzcond9}  is assumed.}
    \end{cases}
  \end{equation}
  Then, there exists a universal constant $C$, independent of all parameters, such that the convergence rate
  $\nu$ can be explicitly estimated for the three PDMPs as
  \begin{equation}\label{eqn:rate}
    \nu\ge C
    \begin{cases}
      \dfrac{m\gamma}{(\sqrt{m}+R+\gamma)^2}, & \text{ for
        RHMC;} \\
      \dfrac{m\gamma}{(\sqrt{m}+R_{ZZ}+\gamma)^2}, & \text{ for ZZ;} \\
      \dfrac{m\gamma}{(\sqrt{dm}+R\sqrt{d}+\gamma)^2}, &
      \text{ for BPS}.
    \end{cases}
  \end{equation}
  Here $R_{ZZ}=\sqrt{L}$ if $\norm{\nabla_x^2 U}\le L, \forall x$ and
  $R_{ZZ}=M\sqrt{d}$ otherwise.
\end{theorem}
\noindent We will prove this theorem in Section~\ref{sec:proofs}.

Given the expression of \eqref{eqn:rate}, we can choose the optimal
$\gamma$ to maximize the rate $\nu$ for the three PDMPs:\begin{equation}
  \gamma =
  \begin{cases}
    \sqrt{m}+R, & \text{ for RHMC;} \\
    \sqrt{m}+R_{ZZ}, & \text{ for ZZ;} \\
    \sqrt{dm}+R\sqrt{d}, & \text{ for BPS.}
  \end{cases}
\end{equation} 
Therefore the optimal convergence rate is given by \begin{equation}
  \nu \ge C
  \begin{cases}
    \dfrac{m}{\sqrt{m}+R}, & \text{ for RHMC;} \\
    \dfrac{m}{\sqrt{m}+R_{ZZ}}, & \text{ for ZZ;} \\
    \dfrac{m}{\sqrt{dm}+R\sqrt{d}}, & \text{ for BPS.}
  \end{cases}
\end{equation}
Table~\ref{tab:rate} summarizes the result under the assumption
$m\id \le \nabla_x^2 U \le L \id$ (and hence guarantee Assumptions
\ref{assump:poincare}-\ref{assump:spectral}) in the most interesting
regime $m\ll 1 \ll L$, with optimal choice of $\gamma$.
\begin{table}[htpb]
  \centering
  \begin{tabular}{c|c|c}
    & convergence rate $\nu$ & optimal $\gamma$ \\[.5mm] \hline\hline
    RHMC  & $O(\sqrt{m})$ & $\sqrt{m}$ \\ \hline ZZ & $O(\dfrac{m}{\sqrt{L}})$ & $\sqrt{L}$ \\ \hline BPS & $O(\sqrt{\dfrac{m}{d}})$ & $\sqrt{dm}$
  \end{tabular}
  \caption{Summary of lower bound on the convergence rate $\nu$ and optimal choice of
    $\gamma$ depending on $d,m,L$ under the assumption
    $m\id \le \nabla_x^2 U \le L \id$ for the regime $m\ll 1 \ll L$.}
  \label{tab:rate}
\end{table}

Compared to \cite{andrieu2018hypocoercivity}, 
we are able to derive an explicit scaling of $\nu$ not only on $d$, but also explicitly on 
$m, L$ as well. For RHMC, we obtain the optimal convergence rate $O(\sqrt{m})$, which is the same as for the underdamped Langevin dynamics \cite{cao2019explicit}. The $O(\sqrt{m})$ rate is optimal as can be checked for the Gaussian case  $U(x)=\frac{m \abs{x}^2}{2}$. For zigzag process, 
we are able to derive dimension independent convergence rate with the 
smoothness assumption $\|\nabla_x^2 U\| \le L$, which is more quantitative than the result in 
\cite{andrieu2018hypocoercivity}. 
Finally, although we are
unable to obtain a dimension independent rate for BPS, our rate $O(d^{-1/2})$ under
the assumption $\nabla_x^2 U \ge -L\id$ is still an
improvement from the rate in \cite{andrieu2018hypocoercivity}, whose estimate provides a rate of $O(d^{-(1+\omega)/2})$ under the assumption $\Delta_x U(x) \le cd^{1+\omega}+|\nabla_x U(x)|^2/2$. It is unclear whether a dimension independent convergence rate is
possible for BPS.

\smallskip

Before we move on to the proof of Theorem \ref{thm:expconvfull}, let us give a brief introduction on the strategy of the proof. A naive energy estimate yields \begin{align*}
    \dfrac{\ud}{\ud t} \|f(t, \cdot)\|_{L^2(\stationary)}^2 & =-\dfrac{1}{2}\sum_{k=1}^K \int_{(s,t)\times\RR^d\times \RR^d}
      |v\cdot F_k| (f-\B_k f)^2 \ud t\ud \stationary \\ & \qquad -2\gamma
      \int_{(s,t)\times\RR^d\times \RR^d}
      (f-\Pi_v f)^2 \ud t\ud \stationary.
\end{align*} 
While the above establishes the $L^2$ energy decay, it does not directly yield exponential decay rate.
In particular, the energy dissipation is only present in velocity variable. However, instead of looking at single time layers, we should look at time intervals, since after time propagation, the dissipation in $v$ together with the transport and bouncing terms in $x$ will lead to dissipation in $x$. With the help a Poincar\'e-type inequality in the augmented state space established in Theorem \ref{thm:bcpoincare}, we can prove exponential convergence still using the standard energy estimate, with $\gamma(\Pi_v-\id)$ playing the role of ``dissipation'', in line with the moral ``hypocoercivity is simply coercivity with respect to the correct norm'', quoted from \cite[Page 4]{armstrong2019variational}. 

\section{Proofs}\label{sec:proofs}

We first state the following modified Poincar\'{e}-type inequality
that generalizes \cite[Theorem 1.2]{armstrong2019variational} and
\cite[Theorem 2]{cao2019explicit} to the PDMP dynamics under
consideration.
\begin{theorem}\label{thm:bcpoincare}
 There exists a constant $C$
  independent of all parameters such that for all functions $f\in C^1_b(I\times\RR^d\times\RR^d)$, 
  \begin{multline}\label{eqn:hypopoincare}
    \|f-(f)_{\lambda\times \stationary}\|_{L^2(\lambda\times \stationary)} \le
\bigl(1+C_{\J}\bigr) \bigl\lVert f-\Pi_v f\bigr\rVert_{L^2(\lambda\times \stationary)}
    \\ +C(\dfrac{1}{\sqrt{m}}+T) \bigl\lVert \partial_t f - v\cdot\nabla_x f-
    \sum_{k=1}^K (v\cdot F_k)_+(\B_k-I)f+F_0\cdot \nabla_v f
    \bigr\rVert_{L^2(\lambda\times \stationary)}.
  \end{multline} 
  Here $C_\J$ is a constant defined as
  \begin{equation}\label{eqn:magcj}
    C_\J = \begin{cases}
      C\Bigl(1+\dfrac{1}{\sqrt{m}T}+\dfrac{R}{\sqrt{m}}+RT\Bigr), & \text{ for RHMC;} \\
      C\Bigl(1+\dfrac{1}{\sqrt{m}T}+\dfrac{R_{ZZ}}{\sqrt{m}}+R_{ZZ}T\Bigr), & \text{ for ZZ;} \\
      C\sqrt{d}\Bigl(1+\dfrac{1}{\sqrt{m}T}+\dfrac{R}{\sqrt{m}}+RT\Bigr), & \text{ for BPS,}
    \end{cases}
  \end{equation} where $R$ and $R_{ZZ}$ are the same quantities as defined in Theorem \ref{thm:expconvfull}.
\end{theorem}
\begin{remark}\label{rmk:extreg}
  We remark that Theorem \ref{thm:bcpoincare} also applies to any $f(t,x,v)$ which is the solution to \eqref{eqn:pdmpeqn} with initial condition $f_0\in L^2(\stationary)$. Since $\opL$ generates a contraction semigroup, for $f(t,\cdot,\cdot)=P_t f_0$ we have $P_t f_0 \in L^2(\stationary)$, and therefore for any fixed t, \begin{equation*}
      \partial_t f - v\cdot\nabla_x f-
    \sum_{k=1}^K (v\cdot F_k)_+(\B_k-I)f+F_0\cdot \nabla_v f = \gamma (\Pi_v f-f) \in L^2(\stationary).
  \end{equation*}
  Hence the right-hand side of \eqref{eqn:hypopoincare} is finite for any $f$ being the solution to \eqref{eqn:pdmpeqn}, and therefore Theorem \ref{thm:bcpoincare} holds for $f$ by density argument.
\end{remark}

To prove Theorem \ref{thm:bcpoincare}, we need the following lemma,
established in our previous work \cite{cao2019explicit}, which
provides crucial test functions that satisfy a divergence equation
with Dirichlet boundary conditions. 
\begin{lemma}[\cite{cao2019explicit}*{Lemma 2.6}]\label{lem:truetestfn}
  For any function $f=f(t,x)\in L^2(\lambda\times \mu_U)$ with
  $(f)_{\lambda\times \mu_U}=0$, there exist
  $\wb{\phi}=(\phi_0,\phi_1,\cdots,\phi_d)^{\top} \in
  H^1(\lambda\times \mu_U)^{d+1}$  solving
  \begin{equation}\label{eqn:divergence}
    -\partial_t\phi_0+\sum_{i=1}^d \partial_{x_i}^*\phi_i = f, \ \ \phi(t=0,\cdot)=\phi(t=T,\cdot)=0,
  \end{equation}
  with estimates
  \begin{align}\label{eqn:estphi}
    \|\wb{\phi}\|_{L^2(\lambda\times \mu_U)} & \le C \max\bigl\{\dfrac{1}{\sqrt{m}},T\bigr\}\| f\|_{L^2(\lambda\times \mu_U)} \\
  \label{eqn:estdiv}
    \Bigl( \sum_{i,j = 0}^d \lVert\partial_{x_i} \phi_j \rVert^2_{L^2(\lambda\times \mu_U)}  \Bigr)^{1/2} & \le C \bigl(1+\dfrac{1}{\sqrt{m}T}+\dfrac{R}{\sqrt{m}}+RT\bigr) \| f\|_{L^2(\lambda\times \mu_U)}.
  \end{align}
  Here $C$ is some universal constant, and $R$ is the ``convexity barrier'' parameter for potential $U$ defined in Theorem~\ref{thm:expconvfull}. 
\end{lemma}

Before proceeding to the proof of Theorem \ref{thm:bcpoincare}, we
present two elementary but useful lemmas: one regarding the properties
of reflections $\B_k$, and the other on intergrating the $v$ variable
with $(v\cdot n)_+$.

\begin{lemma}\label{lem:bk}
  The operators $\B_k$ defined in \eqref{eqn:bk} satisfy the following
  properties:
  \begin{enumerate}
  \item for any functions $f,g$, \[\B_k (fg)=\B_k f\B_k g;\]
  \item $\B_k^2=\id$;
  \item $\B_k$ is symmetric in $L^2(\kappa)$: For any two
    functions
    $f,g$, \begin{equation} \label{eqn:bksym}\int_{\RR^d} \B_k f g
      \ud \kappa(v)= \int_{\RR^d} f\B_k g \ud
      \kappa(v);\end{equation} as a direct consequence, letting
    $g=1$, we have for any function $f$,
    \[\int_{\RR^d} \B_k f \ud \kappa(v)= \int_{\RR^d} f\ud
      \kappa(v);\]
   \item for any function $f$, \begin{equation}\label{eqn:bkvdfkp}
       \int_{\RR^d} (v\cdot F_k)_+ \B_k f \ud \kappa(v)=  \int_{\RR^d} (-v\cdot F_k)_+  f \ud \kappa(v).
   \end{equation}   
  \end{enumerate}
\end{lemma}
\begin{proof}
  The first and second properties can be verified directly using
  definition \eqref{eqn:bk}. The third property follows from a change
  of variable $\tilde{v}:=v-2(v\cdot n_k)n_k$ in $v$, so that $v=\tilde{v}-2(\tilde{v}\cdot n_k)n_k$, and and $\kappa(\tilde{v})=\kappa(v)$:
  \begin{align*}
    \int_{\RR^d} \B_k f g \ud \kappa(v) & =\int_{\RR^d} f(v-2(v\cdot n_k)n_k)g(v)\ud \kappa(v) \\ & = \int_{\RR^d} f(\tilde{v})g(\tilde{v}-2(\tilde{v}\cdot n_k)n_k) \ud \kappa(\tilde{v}) = \int_{\RR^d} f\B_k  g \ud \kappa(v). 
  \end{align*}
  Finally for the fourth property, we change the variables in the same way as the proof of the third one, so that $v\cdot F_k=-\tilde{v}\cdot F_k$:  \begin{align*}
       \int_{\RR^d} (v\cdot F_k)_+ \B_k f \ud \kappa(v) &=  \int_{v\cdot F_k\ge 0} (v\cdot F_k) f(v-2(v\cdot n_k)n_k) \ud \kappa(v) \\ & = \int_{\tilde{v}\cdot F_k\le 0} -(\tilde{v}\cdot F_k) f(\tilde{v}) \ud \kappa(\tilde{v}) \\ & =  \int_{\RR^d} (-v\cdot F_k)_+  f \ud \kappa(v) . \qedhere
  \end{align*}
\end{proof}

\begin{lemma}\label{lem:vinteg}
  For any vector 
  $q\in \RR^d$ and any two functions $\varphi(v\cdot q)$ and
  $\psi(v)$ such that $\varphi(v\cdot q)\psi(v)$ is even in $v$ and $\varphi(0)=0$, it holds
   \begin{equation}\label{eqn:vinteg}
    \int_{\RR^d} \varphi((v\cdot q)_+)\psi(v)\ud\kappa(v)=\dfrac{1}{2} \int_{\RR^d} \varphi(v\cdot q)\psi(v)\ud\kappa(v).
\end{equation}  
\end{lemma}
\begin{proof}
The identity is obtained as follows, in which we use a change of variables $v\mapsto -v$ in the second equality, the symmetry of Gaussian $\kappa(v)$ in the sense that $\kappa(v)=\kappa(-v)$ in the third equality, and $\varphi(0)=0$ in the last equality:
\begin{align*}
  \int_{\RR^d} \varphi(v\cdot q)\psi(v)\ud\kappa(v) & = \int_{v\cdot q\ge 0} \varphi(v\cdot q)\psi(v)\ud\kappa(v)+\int_{v\cdot q \le 0} \varphi(v\cdot q)\psi(v)\ud\kappa(v) \\ & = \int_{v\cdot q\ge 0} \varphi(v\cdot q)\psi(v)\ud\kappa(v)+\int_{v\cdot q \ge 0} \varphi(-v\cdot q)\psi(-v)\ud\kappa(-v)\\ & = 2\int_{v\cdot q\ge 0} \varphi(v\cdot q)\psi(v)\ud\kappa(v)  = 2\int_{\RR^d} \varphi\bigl((v\cdot q)_+\bigr)\psi(v)\ud\kappa(v). \qedhere
\end{align*}
\end{proof}

We are now ready to prove Theorem~\ref{thm:bcpoincare}.
\begin{proof}[Proof of Theorem \ref{thm:bcpoincare}]
  Without loss of generality we assume
  \[(f)_{\lambda\times \stationary}=0.\] We now take
  $\wb{\phi} = (\phi_0, \phi_1, \cdots, \phi_d)^{\top}$ to be the test
  functions given by Lemma \ref{lem:truetestfn} with $\Pi_v f$ playing the role of $f$. 

  Define (for simplicity of notation, we denote
  $\phi = (\phi_1, \cdots, \phi_d)^{\top}$ and treat $\phi$ as a
  $d$-vector)
  \begin{equation}\label{eqn:defj}
    \J:= -\partial_t
    \phi_0+v\cdot\nabla_x\phi_0+v\cdot\partial_t
    \phi-\sum_{i=1}^d v_iv\cdot\partial_{x_i}\phi+ F_0 \cdot \phi
    -2\sum_{k=1}^K(-v\cdot F_k)_+(v\cdot n_k)(\phi\cdot n_k).
  \end{equation}
We claim the following estimate, the proof of which will be deferred:
  \begin{lemma}\label{lem:cj} The quantity $\J$ can be controlled by
    $\Pi_v f$ in the sense of
   \begin{equation}\label{eqn:cj}
     \|\J\|_{L^2(\lambda\times \stationary)} \le C_{\J} \|\Pi_v f\|_{L^2(\lambda\times \mu_U)}.
   \end{equation}
   Here $C_\J$ is the constant defined in Theorem \ref{thm:bcpoincare}.
 \end{lemma}
 Before proceeding with the proof of Theorem \ref{thm:bcpoincare}, let
 us provide a heuristic justification for Lemma \ref{lem:cj}: if we
 calculate $\|\J\|_{L^2(\lambda\times \stationary)}^2$, then its expression
 consists of terms that are up to the fourth moment of $v$ multiplied
 with $\phi_k$, $\partial_{x_i}\phi_j$ or $\phi_k \nabla_x
 U$. Therefore, integrating out the $v$ component against Gaussian,
 and by Lemma \ref{lem:truetestfn} all terms can be controlled by
 $\|\Pi_v f\|_{L^2(\lambda\times \mu_U)}^2$. The actual constants will be
 estimated separately for each PDMP in later part of the paper.
		
 Now let us return to the proof of Theorem \ref{thm:bcpoincare} assuming
 Lemma~\ref{lem:cj}. To simplify notations, we define the operator
  \begin{equation}\label{eqn:opA}
    \A f := \partial_t f - v\cdot\nabla_x f- \sum_{k=1}^K(v\cdot  F_k)_+(\B_k-I)f + F_0 \cdot \nabla_v f.
  \end{equation}
  We now estimate the $L^2$ norm of $\Pi_v f$. Using Lemma
  \ref{lem:vinteg} for $q=-F_k$, $\varphi(v\cdot q)=v\cdot q$ and
  $\psi(v)=(v\cdot n_k)(\phi\cdot n_k)$ and integrating out $v$, we
  have \begin{equation}\label{eqn:thm2vint}\begin{aligned}
      2\int_{I\times \RR^d\times \RR^d}(-v\cdot F_k)_+(v\cdot
      n_k)(\phi\cdot n_k) \ud t \ud \stationary &=- \int_{I\times
        \RR^d\times \RR^d}(v\cdot F_k)(v\cdot n_k)(\phi\cdot n_k) \ud
      t \ud \stationary\\ & = - \int_{I\times \RR^d}\phi\cdot F_k \ud
      t \ud \mu_U(x).
\end{aligned}\end{equation} Therefore, by the construction of
  the test functions $\bar{\phi}$, and noticing $\nabla_x U=\sum_{k=0}^K F_k$, we have
    \begin{align*}
      \|\Pi_v f\|_{L^2(\lambda\times \mu_U)}^2 &=\int_{I\times\RR^d} \Pi_v
      f(-\partial_t \phi_0-\nabla_x\cdot\phi+\phi\cdot \sum_{k=0}^K
      F_k)\ud t\ud\mu_U(x)  \\ & \leftstackrel{\eqref{eqn:thm2vint}}{=}\int_{I\times\RR^d\times\RR^d} \Pi_v
      f\Bigl(-\partial_t \phi_0-\nabla_x\cdot\phi+\phi\cdot F_0\\ & \hspace{8em} -2\sum_{k=1}^K (-v\cdot F_k)_+(v\cdot n_k)(\phi\cdot
      n_k)\Bigr)\ud t\ud\stationary \\ & = \int_{I\times \RR^d\times \RR^d}
      \Pi_v f \Bigl(-\partial_t
      \phi_0+v\cdot\nabla_x\phi_0+v\cdot\partial_t
      \phi-\sum_{i}v_iv\cdot\partial_{x_i}\phi+\phi\cdot F_0 \\ &
      \hspace{8em} -2\sum_{k=1}^K (-v\cdot F_k)_+(v\cdot n_k)(\phi\cdot
      n_k)\Bigr) \ud t \ud \stationary \\ &
      \leftstackrel{\eqref{eqn:defj}}{=} \int_{I\times \RR^d\times
        \RR^d} f\J\ud t \ud \stationary -\int_{I\times \RR^d\times
        \RR^d} (f-\Pi_v f)\J\ud t \ud \stationary \\ &
      \leftstackrel{\eqref{eqn:cj}}{\le} \int_{I\times \RR^d\times
        \RR^d} f\J\ud t \ud \stationary +C_{\J}\|\Pi_v
      f\|_{L^2(\lambda\times \mu_U)}\| f-\Pi_v
      f\|_{L^2(\lambda\times \stationary)}, \stepcounter{equation}\tag{\theequation}\label{eqn:fkappaest}
    \end{align*}
  where the third equality follows from introducing a dummy $v$
  variable and noting that by the basic properties of Gaussian measure
  $\kappa$, $\int_{\RR^d} v_i\ud \kappa(v)=0$, and
  $\int_{\RR^d} v_iv_j\ud \kappa(v)=\delta_{ij}$. 
  
  To estimate the first term on the RHS of \eqref{eqn:fkappaest} we use
  integration by parts. For time derivatives, we crucially use that $\wb{\phi}$ vanish at both boundaries $t=0$ and $t=T$; for derivatives in space and velocity, we use $\int_{\RR^d} \nabla_x f \cdot g \ud \mu_U(x) = \int_{\RR^d}  f (-\nabla_x g+g\cdot \nabla_x U) \ud \mu_U(x)$ and $\int_{\RR^d} \nabla_v f \cdot g \ud \kappa(v) = \int_{\RR^d}  f (-\nabla_v g+g\cdot v) \ud \kappa(v)$:
  \begin{align*}
    & \int_{I\times \RR^d\times \RR^d} f \J \ud t \ud \stationary \\ & \qquad \leftstackrel{\eqref{eqn:defj}}{=}\int_{I\times \RR^d\times \RR^d} f\Bigl(-\partial_t
    \phi_0+v\cdot\nabla_x\phi_0+v\cdot\partial_t
    \phi-\sum_{i=1}^d v_iv\cdot\partial_{x_i}\phi+ F_0 \cdot \phi
    \\ & \qquad \qquad -2\sum_{k=1}^K(-v\cdot F_k)_+(v\cdot n_k)(\phi\cdot n_k)\Bigr) \ud t \ud \stationary\\
    & \qquad =  \int_{I\times
      \RR^d\times \RR^d} \Bigl(\partial_t f \phi_0- v\cdot\nabla_x
      f\phi_0 +f\phi_0 v\cdot \sum_{k=0}^K F_k -\partial_t f
      v\cdot\phi+(v\cdot \nabla_x f)(v\cdot \phi) \\
    & \hspace{8em} - (v\cdot \nabla_x U) (v\cdot \phi) f + f\phi\cdot F_0 -2\sum_{k=1}^K (-v\cdot F_k)_+(v\cdot n_k)(\phi\cdot n_k)f\Bigr) \ud t \ud
      \stationary \\
    & \qquad =\int_{I\times \RR^d\times \RR^d}
      \Bigl((\partial_t f - v\cdot\nabla_x f)(\phi_0-v\cdot \phi)
      +f\phi_0 v\cdot F_0 +f\phi\cdot F_0 - (v\cdot F_0) (v\cdot \phi) f \\
    & \hspace{8em} + (v\cdot \sum_{k=1}^K F_k)(\phi_0-v\cdot
      \phi)f -2\sum_{k=1}^K(-v\cdot F_k)_+(v\cdot n_k)(\phi\cdot n_k)f
      \Bigr) \ud t \ud \stationary \\
    & \qquad =\int_{I\times
      \RR^d\times \RR^d} \Bigl(\bigl(\partial_t f - v\cdot\nabla_x f+
      \sum_{k=1}^K (v\cdot F_k)_+f\bigr)(\phi_0-v\cdot \phi) +\phi_0
      F_0\cdot \nabla_v f  -(F_0\cdot\nabla_v f ) (v\cdot \phi) \\
    & \hspace{8em} - \sum_{k=1}^K(-v\cdot F_k)_+(\phi_0-v\cdot \phi)f
      -2\sum_{k=1}^K(-v\cdot F_k)_+(v\cdot n_k)(\phi\cdot n_k)f\Bigr)
      \ud t \ud \stationary \\
    & \qquad =\int_{I\times \RR^d\times
      \RR^d} \Bigl(\bigl(\partial_t f - v\cdot\nabla_x f+ \sum_{k=1}^K
      (v\cdot F_k)_+f+F_0\cdot \nabla_v f  \bigr)(\phi_0-v\cdot \phi) \\
    & \hspace{8em} -\sum_{k=1}^K(-v\cdot F_k)_+\B_k(\phi_0-v\cdot
      \phi) \B_k^2f\Bigr) \ud t \ud \stationary \\
    & \qquad\leftstackrel{\eqref{eqn:bksym},\eqref{eqn:bkvdfkp}}{=}\int_{I\times
      \RR^d\times \RR^d} \Bigl(\bigl(\partial_t f - v\cdot\nabla_x f-
      \sum_{k=1}^K (v\cdot F_k)_+(\B_k-I)f+F_0\cdot \nabla_v f
      \bigr)(\phi_0-v\cdot \phi) \Bigr) \ud t \ud \stationary \\
    & \qquad \leftstackrel{\eqref{eqn:opA}}{\le} \|\phi_0-v\cdot
      \phi\|_{L^2(\lambda\times \stationary)}\|\A f\|_{L^2(\lambda\times
      \stationary)}.
  \end{align*}
  Since $\wb{\phi}$ is independent of $v$, expanding and integrating
  out $v$ with respect to $\kappa$,  we obtain
  \begin{equation*}
    \|\phi_0-v\cdot \phi\|_{L^2(\lambda\times \stationary)}^2
    =  \|\wb{\phi}\|_{L^2(\lambda\times \mu_U)}^2
    \leftstackrel{\eqref{eqn:estphi}}{\le} C(\dfrac{1}{\sqrt{m}}+T)^2
    \|\Pi_v f\|_{L^2(\lambda\times \mu_U)}^2.
  \end{equation*}
  Thus 
  \begin{equation}\label{eqn:thnfj}
    \int_{I\times \RR^d\times  \RR^d} f \J \ud t \ud \stationary  \le C(\dfrac{1}{\sqrt{m}}+T) \|\Pi_v f\|_{L^2(\lambda\times \mu_U)}\|\A f \|_{L^2(\lambda\times \stationary)}.
  \end{equation}
  Combining \eqref{eqn:fkappaest} and \eqref{eqn:thnfj} we arrive at
  \begin{equation*}
    \|\Pi_v f\|_{L^2(\lambda\times \mu_U)} \le C(\dfrac{1}{\sqrt{m}}+T)\|\A f \|_{L^2(\lambda\times \stationary)} + C_\J \| f-\Pi_v f\|_{L^2(\lambda\times \stationary)},
  \end{equation*}
  and therefore by triangle inequality
  \begin{align*}
    \|f\|_{L^2(\lambda\times \stationary)} & \le \|f-\Pi_v
    f\|_{L^2(\lambda\times \stationary)} +\|\Pi_v f\|_{L^2(\lambda\times \mu_U)} \\
    & \le C(\dfrac{1}{\sqrt{m}}+T)\|\A f \|_{L^2(\lambda\times \stationary)}
    + (1+C_\J) \| f-\Pi_v f\|_{L^2(\lambda\times \stationary)}. \qedhere
  \end{align*}
\end{proof}

\begin{proof}[Proof of Theorem \ref{thm:expconvfull}]
  We first notice that 
  \begin{equation*}
    \int_{\mathbb{R}^d\times \mathbb{R}^d} f(t,x,v) \ud
    \stationary(x,v)=0. 
  \end{equation*}
  for all $t>0$. In view of \eqref{eqn:meanzero}, it suffices to prove
  \begin{equation*}
    \dfrac{\ud}{\ud t}\int_{\mathbb{R}^d\times \mathbb{R}^d} f(t,x,v) \ud \stationary(x,v) \leftstackrel{\eqref{eqn:pdmpeqn}}{=}\int_{\mathbb{R}^d\times \mathbb{R}^d} \opL f(t,x,v) \ud \stationary(x,v)=0.
  \end{equation*}Notice \begin{equation*}
      \opL = v\cdot \nabla_x -\nabla_x U\cdot \nabla_v + \sum_{k=1}^K \opL_k +\gamma(\Pi_v-\id),
  \end{equation*} where $     \opL_k=(v\cdot F_k)_+(\B_k-\id)+F_k\cdot\nabla_v$.
  Therefore, as it is clear that both operators $v\cdot \nabla_x -\nabla_x U\cdot \nabla_v$ and $\Pi_v-\id$ preserve $\stationary$, it suffices to show that $\opL_k$ preserves $\stationary$ as well. Indeed, \begin{align*}
      \int_{\mathbb{R}^d\times \mathbb{R}^d} \opL_k f \ud \stationary(x,v)
    & = \int_{\mathbb{R}^d\times \mathbb{R}^d}  \Bigl((v\cdot F_k)_+\B_k f-(v\cdot F_k)_+f +F_k\cdot \nabla_v f\Bigr)  \ud \stationary \\
    & \leftstackrel{\eqref{eqn:bkvdfkp}}{=} \int_{\mathbb{R}^d\times \mathbb{R}^d} \Bigl( (-v\cdot F_k)_+  f-(v\cdot F_k)_+ f+fv\cdot F_k\Bigr) \ud \stationary =0.
  \end{align*}

  Next we establish the energy decay properties of $f$. Take any two positive numbers $0<s<t$. Following
  \cite[Proposition 8]{andrieu2018hypocoercivity}, we denote the
  symmetric part of $\opL$ by
  \begin{equation}\label{eqn:symmetricl}
    \mathcal{S} = \dfrac{1}{2}\sum_{k=1}^K |v\cdot F_k|(\B_k-\id)
    +\gamma (\Pi_v-\id).
  \end{equation}
  Using the properties of $\B_k$ in Lemma \ref{lem:bk},
  \begin{align*}
    \int_{(s,t)\times\RR^d\times \RR^d} |v\cdot F_k|
    (\B_k f)^2 \ud t\ud \stationary
    & = \int_{(s,t)\times\RR^d\times \RR^d} |v\cdot F_k|
      \B_k f^2 \ud t\ud \stationary \\
    & \leftstackrel{\eqref{eqn:bksym}}{=} \int_{(s,t)\times\RR^d\times \RR^d}
      \B_k|v\cdot F_k|  f^2 \ud t\ud \stationary \\
    & = \int_{(s,t)\times\RR^d\times \RR^d} |v\cdot F_k| f^2 \ud t\ud
      \stationary.
  \end{align*}
  Therefore
  \begin{equation}\label{eqn:squareB}
    \int_{(s,t)\times\RR^d\times \RR^d}
    |v\cdot F_k| (f-\B_k f)^2 \ud t\ud \stationary = 2
    \int_{(s,t)\times\RR^d\times \RR^d} |v\cdot F_k| f(\id-\B_k )f \ud
    t\ud \stationary.
  \end{equation}
  On the other hand, since \begin{equation*}
       \int_{(s,t)\times\RR^d\times \RR^d} f\Pi_v  f \ud t\ud \stationary =
    \int_{(s,t)\times\RR^d\times \RR^d} (\Pi_v  f)^2 \ud t\ud \stationary= \int_{(s,t)\times\RR^d} (\Pi_v  f)^2 \ud t\ud \mu_U(x),
  \end{equation*}
  we have
  \begin{align*}
    \int_{(s,t)\times\RR^d\times \RR^d} (f-\Pi_v  f)^2 \ud t\ud \stationary =
    \int_{(s,t)\times\RR^d\times \RR^d} f(\id-\Pi_v  )f \ud t\ud \stationary.
  \end{align*}
  Therefore we have an elementary energy estimate, noticing the anti-symmetric part of $\opL$ does not contribute to the integral $\int_{(s,t)\times\RR^d\times \RR^d} f\opL f \ud t\ud \stationary$:
  \begin{equation}\label{eqn:energyestimate}
    \begin{aligned}
      \|f&(t, \cdot)\|_{L^2(\stationary)}^2-\|f(s,\cdot)\|_{L^2(\stationary)}^2 = 2\int_{(s,t)\times\RR^d\times
        \RR^d} f\partial_tf \ud t\ud \stationary
      \leftstackrel{\eqref{eqn:symmetricl}}{=}2\int_{(s,t)\times\RR^d\times
        \RR^d} f\mathcal{S} f \ud t\ud \stationary \\
      & = \sum_{k=1}^K \int_{(s,t)\times\RR^d\times \RR^d} |v\cdot
      F_k| f (\B_k-I) f \ud t\ud \stationary +2\gamma
      \int_{(s,t)\times\RR^d\times \RR^d} f
      (\Pi_v-I) f \ud t\ud \stationary \\
      & \leftstackrel{\eqref{eqn:squareB}}{=}
      -\dfrac{1}{2}\sum_{k=1}^K \int_{(s,t)\times\RR^d\times \RR^d}
      |v\cdot F_k| (f-\B_k f)^2 \ud t\ud \stationary -2\gamma
      \int_{(s,t)\times\RR^d\times \RR^d}
      (f-\Pi_v f)^2 \ud t\ud \stationary \\
      & \le - 2\gamma \|f-\Pi_v f\|^2_{L^2(\lambda_{(s,t)}\times \stationary)},
    \end{aligned}
  \end{equation}
 where we use $\lambda_{(s,t)}$ to denote the Lebesgue measure on $(s,t)$. In particular,
  \begin{equation}\label{eqn:nonincL2est}
    \text{ the mapping } t\mapsto \|f(t,\cdot)\|_{L^2(\stationary)}^2 \text{ is nonincreasing.}
  \end{equation}
  By equation \eqref{eqn:pdmpeqn} and definition of the operators \eqref{eqn:opL} and \eqref{eqn:opA},
  \begin{equation}\label{eqn:AfL2}
    \|\mathcal{A}f \|_{L^2(\lambda_{(s,t)}\times\stationary)} =
    \gamma\|f-\Pi_v f\|_{L^2(\lambda_{(s,t)}\times \stationary)}.
  \end{equation}
  Therefore, for any $0<s<t$ (note that Theorem \ref{thm:bcpoincare} applies by Remark \ref{rmk:extreg}), 
  \begin{align*}
    \|f(& t, \cdot)\|_{L^2(\stationary)}^2-\|f(s,\cdot)\|_{L^2(\stationary)}^2\\
        & \leftstackrel{\eqref{eqn:energyestimate}}{\le} - 2 \gamma \norm{ f - \Pi_v f}_{L^2((s, t)\times \stationary)}^2\\
        & \leftstackrel{\eqref{eqn:AfL2}}{=} -\dfrac{2\gamma}{(1+C_\J +C\gamma
          (\frac{1}{\sqrt{m}}+t-s))^2} \Bigl((1+C_\J)\norm{f-\Pi_v f}_{L^2(\lambda_{(s,t)}\times\stationary)}\\
        & \qquad\qquad  +C(\dfrac{1}{\sqrt{m}}+t-s)
          \|\mathcal{A}f \|_{L^2(\lambda_{(s,t)}\times
          \stationary)}\Bigr)^2 \\
        & \leftstackrel{\eqref{eqn:hypopoincare}}{\le}
          -\dfrac{2\gamma}{(1+C_\J+C\gamma (\frac{1}{\sqrt{m}}+t-s))^2}
          \norm{f}_{L^2(\lambda_{(s,t)}\times \stationary)}^2 \\
        & \leftstackrel{\eqref{eqn:nonincL2est}}{\le} -\dfrac{2\gamma(t-s)}{(1+C_\J
          +C\gamma (\frac{1}{\sqrt{m}}+t-s))^2}\norm{f(t, \cdot)}_{L^2( \stationary)}^2.
  \end{align*}
  Now fixing a $T > 0$ to be optimized later, for any $t>0$, we pick
  the integer $k$ satisfying $kT\le t < (k+1)T$. Applying above
  inequality iteratively and using the monoticity
  \eqref{eqn:nonincL2est}, we obtain
  \begin{align*}
    \|f(t,\cdot)\|_{L^2(\stationary)}^2
    & \le \Bigl(1+\dfrac{2\gamma T}{(1+C_\J +C\gamma
      (\frac{1}{\sqrt{m}}+T))^2}\Bigr)^{-k} \|f_0\|_{L^2(\stationary)}^2 \\
    & \le \Bigl(1+\dfrac{2\gamma T}{(1+C_\J +C\gamma
      (\frac{1}{\sqrt{m}}+T))^2}\Bigr)^{-\frac{t}{T}+1}
      \|f_0\|_{L^2(\stationary)}^2 \\
    & \le  \Bigl(1+\dfrac{2\gamma T}{(1+C_\J +C\gamma
      (\frac{1}{\sqrt{m}}+T))^2}\Bigr) \\ & \qquad \times \exp
      \Bigl(-\dfrac{t}{T}\log\bigl(1+\dfrac{2\gamma T}{(1+C_\J
      +C\gamma (\frac{1}{\sqrt{m}}+T))^2}\bigr)\Bigr)
      \|f_0\|_{L^2(\stationary)}^2.
  \end{align*}
  The prefactor
  \begin{align*}
    1+\dfrac{2\gamma T}{(1+C_\J +C\gamma
      (\frac{1}{\sqrt{m}}+T))^2} \le 1+\dfrac{2\gamma T}{(1 +C\gamma
      T)^2} \le 1 + \frac{1}{C}
  \end{align*}
  is bounded above by a universal constant. Therefore, using
  $\log(1+x) \ge \frac{1}{C}x$ for $x \in [0, 1 + \frac{1}{C}]$ for some universal constant $C$, this yields
  \eqref{eqn:expcvg} with the exponential decay rate
  \begin{equation}\label{eqn:expnu}
    \nu=\sup_{T>0}\dfrac{1}{2T}\log \Bigl(1+\dfrac{2\gamma T}{(1+C_\J +C\gamma
      (\frac{1}{\sqrt{m}}+T))^2}\Bigr)
    \ge C\sup_{T>0}\dfrac{\gamma }{(1+C_\J +\gamma
      (\frac{1}{\sqrt{m}}+T))^2}.
  \end{equation}
  Substituting \eqref{eqn:magcj} into \eqref{eqn:expnu}, we get
  \begin{equation}
    \nu \ge C \begin{cases}
     \dfrac{\gamma}{(1+\frac{1}{\sqrt{m}T}+\frac{R}{\sqrt{m}}
        +RT+\frac{\gamma}{\sqrt{m}}+\gamma T)^2}, & \text{ for RHMC;} \\
      \dfrac{\gamma}{(1+\frac{1}{\sqrt{m}T}+\frac{R_{ZZ}}{\sqrt{m}}
        +R_{ZZ}T+\frac{\gamma}{\sqrt{m}}+\gamma T)^2}, & \text{ for ZZ;} \\
      \dfrac{\gamma }{\bigl(\sqrt{d}(1+\frac{1}{\sqrt{m}T}+\frac{R}{\sqrt{m}}
        +RT) +\gamma (\frac{1}{\sqrt{m}}+T)\bigr)^2}, & \text{ for BPS.}
    \end{cases}
  \end{equation}
  We arrive at the rates \eqref{eqn:rate} by optimizing the choice of
  $T$ for each case: \begin{equation*}
      T=\begin{cases}
      m^{-\frac{1}{4}}(R+\gamma)^{-\frac{1}{2}}, & \text{ for RHMC;} \\ m^{-\frac{1}{4}}(R_{ZZ}+\gamma)^{-\frac{1}{2}}, & \text{ for ZZ;} \\ d^\frac{1}{4}m^{-\frac{1}{4}}(\sqrt{d}R+\gamma)^{-\frac{1}{2}}, & \text{ for BPS}. 
      \end{cases} \qedhere
  \end{equation*}
\end{proof}

The remaining task is to prove Lemma \ref{lem:cj}.  For
RHMC,
\[\J=-\partial_t \phi_0+v\cdot\nabla_x \phi_0+v\cdot\partial_t
  \phi-\sum_{i=1}^d v_iv\cdot \partial_{x_i}\phi +\phi\cdot\nabla_x U. \]
The norm $\|\J\|_{L^2(\lambda\times \stationary)}$ is already estimated in
\cite[Proof of Theorem 2]{cao2019explicit}, and the proof is thus
omitted here. In the two subsequent subsections we will estimate
$C_\J$ for ZZ and BPS respectively.

\subsection{The zigzag process}
In this case \[\J=-\partial_t \phi_0+v\cdot\nabla_x\phi_0+v\cdot\partial_t \phi-\sum_{i=1}^d v_iv\cdot\partial_{x_i}\phi -2\sum_{k=1}^d(-v_k\partial_{x_k}U)_+v_k\phi_k. \]
\begin{lemma}\label{lem:zigzag}
Let $\phi_i$, $i=0,\cdots,d$ be test functions as in Lemma \ref{lem:truetestfn}. Then \begin{equation}\label{eqn:zigzag}
        \sum_{k=1}^d \int_{I\times\RR^d} (\phi_k\partial_{x_k} U)^2 \ud t\ud \mu_U(x) \le C \bigl(1+\dfrac{1}{\sqrt{m}T}+\dfrac{R_{ZZ}}{\sqrt{m}}+R_{ZZ}T\bigr)^2\int_{I\times \RR^d} (\Pi_v f)^2 \ud t \ud \mu_U(x).
    \end{equation}Here $R_{ZZ}$ is defined as in Theorem \ref{thm:expconvfull}.
\end{lemma}
\begin{proof}
  Using integration by parts,
  \begin{align*}
    \sum_{k=1}^d \int_{I\times\RR^d}
    & (\phi_k\partial_{x_k} U)^2 \ud t\ud \mu_U(x)=  \sum_{k=1}^d \int_{I\times\RR^d} \partial_{x_k}(\phi_k^2\partial_{x_k} U)  \ud t\ud \mu_U(x) \\
    & = 2\sum_{k=1}^d \int_{I\times\RR^d} \phi_k\partial_{x_k}\phi_k\partial_{x_k} U  \ud t\ud \mu_U(x)+ \sum_{k=1}^d \int_{I\times\RR^d} \phi_k^2\partial_{x_k x_k} U  \ud t\ud \mu_U(x) \\
    & \le \sum_{k=1}^d \int_{I\times\RR^d}  \bigl(\dfrac{1}{2}(\phi_k\partial_{x_k} U)^2
      + 2 |\partial_{x_k}\phi_k|^2\bigr)\ud t\ud \mu_U(x)
      + \sum_{k=1}^d \int_{I\times\RR^d} \phi_k^2\partial_{x_k x_k} U  \ud t\ud \mu_U(x).
  \end{align*}
  After rearranging, we have
  \begin{multline}\label{eqn:zzlem}
    \sum_{k=1}^d \int_{I\times\RR^d} (\phi_k\partial_{x_k} U)^2 \ud
    t\ud \mu_U(x)\le C\sum_{k=1}^d \int_{I\times\RR^d}
    (|\partial_{x_k}\phi_k|^2+\phi_k^2\partial_{x_k x_k} U ) \ud t\ud
    \mu_U(x) \\
    \leftstackrel{\eqref{eqn:estdiv}}{\le}
    C(1+\dfrac{1}{\sqrt{m}T}+\dfrac{R}{\sqrt{m}}+RT)^2 \int_{I\times
      \RR^d} (\Pi_v f)^2 \ud t \ud \mu_U(x) + C\sum_{k=1}^d
    \int_{I\times\RR^d} \phi_k^2\partial_{x_k x_k} U\ud t\ud \mu_U(x).
   \end{multline}
   We first discuss the easier case where $\norm{\nabla_x^2 U}\le L$:
   \begin{align*}
     \sum_{k=1}^d \int_{I\times\RR^d}  \phi_k^2\partial_{x_k x_k} U\ud t\ud \mu_U(x)
     & \le L \sum_{k=1}^d \int_{I\times\RR^d} \phi_k^2\ud t\ud \mu_U(x) \\
     & \leftstackrel{\eqref{eqn:estphi}}{\le}C L(\dfrac{1}{\sqrt{m}}+T)^2
       \int_{I\times \RR^d} (\Pi_v f)^2 \ud t \ud \mu_U(x).
   \end{align*}
   In the general setting where only Assumption \ref{assump:hessian} is assumed, by \cite[Lemma 2.2]{cao2019explicit}, we have \begin{equation}\label{eqn:udplm22}
       \|\phi_k|\nabla_x
       U|\|_{L^2(\lambda\times \mu_U)} \le C \bigl(\|\nabla_x \phi_k\|_{L^2(\lambda\times \mu_U)}
       + Md\|\phi_k\|_{L^2(\lambda\times \mu_U)}\bigr).
   \end{equation}Therefore, by Cauchy-Schwarz inequality,
   \begin{align*}
     \sum_{k=1}^d \int_{I\times\RR^d}
     & \phi_k^2\partial_{x_k x_k} U\ud t\ud \mu_U(x) \\
     & \leftstackrel{\eqref{eqn:stoltzcond9}}{\le} CM \sum_{k=1}^d
       \int_{I\times\RR^d} (1+|\nabla_x U|) \phi_k^2\ud t\ud \mu_U(x) \\
     & \le CM \|\phi\|_{L^2(\lambda\times \mu_U)}^2 +CM \sum_{k=1}^d
       \|\phi_k\|_{L^2(\lambda\times \mu_U)}\|\phi_k|\nabla_x
       U|\|_{L^2(\lambda\times \mu_U)} \\
     & \leftstackrel{\eqref{eqn:udplm22}}{\le} CM \Bigl(\|\phi\|_{L^2(\lambda\times \mu_U)}^2+\sum_{k=1}^d \|\phi_k\|_{L^2(\lambda\times \mu_U)}
       \bigl(\|\nabla_x \phi_k\|_{L^2(\lambda\times \mu_U)}
       + Md\|\phi_k\|_{L^2(\lambda\times \mu_U)}\bigr)\Bigr) \\
     & \leftstackrel{\eqref{eqn:estphi}}{\le}
       CM\Bigl((\sum_{k=1}^d\|\phi_k\|_{L^2(\lambda\times \mu_U)}^2)^\frac{1}{2}(\sum_{k=1}^d
       \|\nabla_x\phi_k\|_{L^2(\lambda\times \mu_U)}^2)^\frac{1}{2}+Md
       \|\phi\|_{L^2(\lambda\times \mu_U)}^2\Bigr) \\
     & \leftstackrel{\eqref{eqn:estphi},\eqref{eqn:estdiv}}{\le}
       C(1+\frac{1}{\sqrt{m}T}+\frac{M\sqrt{d}}{\sqrt{m}}
       +M\sqrt{d}T)^2\|\Pi_v f\|_{L^2(\lambda\times \mu_U)}^2.
   \end{align*}
 This proves the lemma with $R_{ZZ}=M\sqrt{d}$.
 \end{proof}

 \begin{proof}[Proof of Lemma \ref{lem:cj} for zigzag process]
   To estimate $\| \J \|_{L^2(\lambda\times \stationary)}^2$, we expand its terms, categorize them according to their powers on $v$ and whether $(-v_k\partial_{x_k}U)_+$ is contained, and integrate out the $v$ variable for each term.

   We start with terms that do not contain $(-v_k\partial_{x_k}U)_+$, in which all terms with odd power of $v$ vanish:\\
   Terms with $0$-th power of $v$:
   \begin{equation*}
     \int_{I\times \mathbb{R}^d} (\partial_t \phi_0)^2\ud t\ud\mu_U(x).
   \end{equation*}
   Terms with $2$-nd power of $v$:
    \begin{align*}
       \int_{I\times \mathbb{R}^d\times \mathbb{R}^d}
       & \Bigl((v\cdot\nabla_x \phi_0)^2+(v\cdot\partial_t \phi)^2
         +2(v\cdot\partial_t\phi)(v\cdot\nabla_x\phi_0)+2\sum_{i,j=1}^d v_iv_j\partial_t
         \phi_0\partial_{x_i}\phi_j\Bigr)\ud t\ud \stationary \\
       & = \int_{I\times\mathbb{R}^d}
         \Bigl((\nabla_x\phi_0)^2+(\partial_t\phi)^2
         +2(\partial_t\phi\cdot\nabla_x\phi_0)+2
         \partial_t
         \phi_0\sum_{i=1}^d\partial_{x_i}\phi_i\Bigr)\ud t\ud \mu_U(x) \\
       & \le \int_{I\times\mathbb{R}^d}
         \Bigl(2(\nabla_x \phi_0)^2+2(\partial_t\phi)^2
         +2         \partial_t
         \phi_0\sum_{i=1}^d\partial_{x_i}\phi_i\Bigr)\ud
         t\ud \mu_U(x).
    \end{align*}
    Terms with $4$-th power on $v$: 
    \begin{align*}
          \int_{I\times \mathbb{R}^d\times \mathbb{R}^d} &  \sum_{i,j,p,q=1}^d v_iv_jv_pv_q \partial_{x_i}\phi_j \partial_{x_p}\phi_q \ud t \ud \stationary \\ & =  \int_{I\times \mathbb{R}^d} \Bigl(3\sum_{i=1}^d (\partial_{x_i}\phi_i)^2 + \sum_{1\le i\neq j \le d}\bigl(|\partial_{x_i}\phi_j|^2  + \partial_{x_i}\phi_i\partial_{x_j}\phi_j  +\partial_{x_i}\phi_j\partial_{x_j}\phi_i\bigr) \Bigr)\ud t\ud\mu_U(x) \\ & \le  \int_{I\times \mathbb{R}^d} \Bigl((\sum_{i=1}^d\partial_{x_i}\phi_i)^2 +2 \sum_{i, j=1}^d |\partial_{x_i}\phi_j|^2   \Bigr)\ud t\ud\mu_U(x).
     \end{align*} 
     
     Now we look at the terms with ``$(-v_k\partial_{x_k}U)_+$''. \\
     Terms where $(-v_k\partial_{x_k}U)_+$ appearing twice, in which case the overall power of $v$ is even and thus Lemma~\ref{lem:vinteg} is applicable: 
     \begin{align*}
          \int_{I\times \mathbb{R}^d\times \mathbb{R}^d} & 4\sum_{k,p=1}^d(-v_k \partial_{x_k} U )_+(-v_p \partial_{x_p} U)_+v_kv_p\phi_k\phi_p \ud t\ud\stationary \\ & = \int_{I\times \mathbb{R}^d\times \mathbb{R}^d}  \Bigl(2\sum_{k=1}^d (-v_k \partial_{x_k} U )^2v_k^2\phi_k^2 + \sum_{1\le k\neq p \le d} v_k^2v_p^2\phi_k\partial_{x_k}U \phi_p \partial_{x_p}U \Bigr) \ud t\ud\stationary \\ & = \int_{I\times \mathbb{R}^d}  \Bigl(6\sum_{k=1}^d \phi_k^2 ( \partial_{x_k} U )^2+ \sum_{1\le k\neq p\le d} \phi_k\partial_{x_k}U \phi_p \partial_{x_p}U \Bigr) \ud t\ud\mu_U(x)  \\ & = \int_{I\times \mathbb{R}^d}  \Bigl((\sum_{k=1}^d \phi_k  \partial_{x_k} U )^2+ 5\sum_{k=1}^d  \phi_k^2(\partial_{x_k}U)^2 \Bigr) \ud t\ud\mu_U(x).
     \end{align*}
     Cross terms with $(-v_k\partial_{x_k} U)_+$ where we could still use Lemma \ref{lem:vinteg} due to an overall even power of $v$: \begin{align*}
         \int_{I\times \mathbb{R}^d\times \mathbb{R}^d} & \Bigl(4\partial_t\phi_0\sum_{k=1}^d (-v_k \partial_{x_k} U)_+v_k \phi_k  +4\sum_{i,j,k=1}^d v_iv_j \partial_{x_i} \phi_j(-v_k \partial_{x_k} U)_+v_k \phi_k\Bigr)\ud t\ud\stationary \\ & = \int_{I\times \mathbb{R}^d\times \mathbb{R}^d}  \Bigl(-2\partial_t\phi_0\sum_{k=1}^d v_k^2 \partial_{x_k} U \phi_k  -2\sum_{i=1}^d v_i^4 \partial_{x_i} \phi_i \partial_{x_i} U\phi_i \\ & \qquad -2\sum_{1\le i\neq j\le d} v_i^2v_j^2 \partial_{x_i}\phi_i \partial_{x_j}U \phi_j\Bigr)\ud t\ud\stationary  \\ & = \int_{I\times \mathbb{R}^d}  \Bigl(-2\partial_t\phi_0\sum_{k=1}^d  \partial_{x_k} U \phi_k  -6\sum_{i=1}^d \partial_{x_i} \phi_i \partial_{x_i} U\phi_i-2\sum_{1\le i\neq j\le d}  \partial_{x_i}\phi_i \partial_{x_j}U \phi_j\Bigr)\ud t\ud\mu_U(x)  \\ & \le \int_{I\times \mathbb{R}^d}  \Bigl(-2(\partial_t\phi_0+\sum_{i=1}^d \partial_{x_i}\phi_i)(\sum_{k=1}^d   \phi_k\partial_{x_k} U)  +2\sum_{i=1}^d (|\partial_{x_i} \phi_i|^2+ |\phi_i\partial_{x_i} U|^2)\Bigr)\ud t\ud\mu_U(x).
    \end{align*}                                                                    Finally cross terms with $(-v_k\partial_{x_k} U)_+$ where one cannot use Lemma \ref{lem:vinteg} due to an overall odd power of $v$. In this case, instead of calculating an exact integral (which we actually can, but it does not yield a better bound), for simplicity we control these terms by what we have calculated above: 
    \begin{align*}
        -4 \int_{I\times \mathbb{R}^d\times \mathbb{R}^d} & v\cdot (\partial_t \phi+\nabla_x\phi_0)\sum_{k=1}^d (-v_k\partial_{x_k}U)_+v_k\phi_k \ud t\ud\stationary \\ &  =-4 \int_{I\times \mathbb{R}^d\times \mathbb{R}^d}  \sum_{k=1}^d (\partial_t \phi_k+\partial_{x_k}\phi_0) (-v_k\partial_{x_k}U)_+v_k^2\phi_k \ud t\ud\stationary  \\ & \le  \int_{I\times \mathbb{R}^d\times \mathbb{R}^d} \Bigl( \sum_{k=1}^d v_k^4 (\partial_t \phi_k+\partial_{x_k}\phi_0)^2+4\sum_{k=1}^d (-v_k\partial_{x_k}U)_+^2\phi_k^2\Bigr) \ud t\ud\stationary  \\ & =  \int_{I\times \mathbb{R}^d\times \mathbb{R}^d} \Bigl( 3\sum_{k=1}^d  (\partial_t \phi_k+\partial_{x_k}\phi_0)^2+2\sum_{k=1}^d (v_k\partial_{x_k}U)^2\phi_k^2\Bigr) \ud t\ud\stationary \\ & \le \int_{I\times \mathbb{R}^d} \Bigl( 6\sum_{k=1}^d  \bigl((\partial_t \phi_k)^2+(\partial_{x_k}\phi_0)^2\bigr)+2\sum_{k=1}^d \phi_k^2(\partial_{x_k}U)^2\Bigr) \ud t\ud\mu_U(x).
    \end{align*} 
    Therefore, combining these calculations, we obtain finally  
    \begin{align*}
      \norm{\J}_{L^2(\lambda\times \stationary)}^2
     & \le \int_{I\times \mathbb{R}^d} \Big(|\partial_t\phi_0+\sum_{i=1}^d \partial_{x_i}\phi_i -\sum_{i=1}^d \phi_i \partial_{x_i} U|^2 +8\sum_{i,j=0}^d|\partial_{x_i}\phi_j|^2 \\ & \qquad +9\sum_{k=1}^d(\partial_{x_k} U )^2\phi_k^2 \Big)\ud t\ud\mu_U(x) \\
     & \leftstackrel{\eqref{eqn:divergence},\eqref{eqn:estdiv},\eqref{eqn:zigzag}}{\le} C(1+\dfrac{1}{\sqrt{m}T}+\dfrac{R_{ZZ}}{\sqrt{m}}+R_{ZZ}T)^2\|\Pi_v f\|_{L^2(\lambda\times \mu_U)}^2. \qedhere
   \end{align*}
 \end{proof}
 
 \begin{remark}\label{rmk:tensor}
   Our bound in Lemma \ref{lem:zigzag} can be improved for some
   specific cases. For example, if the potential has a separate form
   $U(x)=\sum_{k=1}^d U_k(x_k)$ with $U_k''(x)\ge -L$ for all $k$, we
   claim the convergence rate $\nu$ is dimension independent,
   regardless of growth condition of $U$, recovering the result in \cite{andrieu2018hypocoercivity}.
   
   For the proof of this, we need to revisit the construction of the test functions $\phi_k$ in the proof of \cite[Lemma 2.6]{cao2019explicit}, and make a more refined estimate than that in Lemma \ref{lem:zigzag}. We will follow the notations of the proof of \cite[Lemma 2.6]{cao2019explicit}. Let us decompose
   \[\Pi_v f=f^\perp +c_0(t-\frac{T}{2})+\sum_{\alpha} (c_\alpha^+
     e^{\alpha t}+c_\alpha^- e^{\alpha(T-t)})w_\alpha(x),\] where
   $c_0,c_\alpha^\pm$ are numbers, $f^\perp$ is perpendicular to all harmonic functions in
   $\lambda\times \mu_U$, in the sense that for any $g\in H^2(\lambda\times \mu_U)$, \begin{equation*}
       \Lap g=-\partial_{tt} g+\nabla_x^*\nabla_x g=0 \ \Rightarrow \  \int_{I\times\RR^d} f^\perp g\ud t\ud\mu_U(x)=0,
   \end{equation*}
and $\alpha^2,w_\alpha$ are
   corresponding eigenvalues and eigenfunctions of
   $\nabla_x^*\nabla_x$:
   \[\nabla_x^*\nabla_x w_\alpha = \alpha^2 w_\alpha, \
     \|w_\alpha\|_{L^2(\mu_U)}=1.\] By linear combination, it suffices
   to prove in both cases $\Pi_v f=f^\perp$ and
   $\Pi_v f=e^{\alpha t}w_\alpha(x)$ (note in the case
   $\Pi_v f=t-\frac{T}{2}$ the corresponding $\phi_k=0$ for $k\ge 1$,
   and thus \eqref{eqn:diag} trivially holds), the corresponding
   functions $\phi_k$ satisfy
   \begin{equation}\label{eqn:diag}
     \sum_{k=1}^d \int_{I\times\RR^d} \phi_k^2\partial_{x_k x_k} U\ud
     t\ud \mu_U(x) \le \|\Pi_v f\|^2_{L^2(\lambda\times \mu_U)}.
   \end{equation}

   First consider the case $\Pi_v f=f^\perp$, $\phi_k=\partial_{x_k} u$ where $u$
   is the solution of the elliptic equation
   \begin{equation}\label{eqn:mixedelliptic}
     \left\{ \begin{aligned} & \Lap
         u=f^\perp & \mbox{in} & \ I\times \mathbb{R}^d,\\ &
         \partial_t u(t=0, \cdot)=\partial_t u(t=T,\cdot)=0 &
         \mbox{in} & \ \mathbb{R}^d.
       \end{aligned}\right.
   \end{equation}
   By Bochner's formula, using the fact
   that $U(x)=\sum_{k=1}^d U_k(x_k)$,
   \begin{align*}
     \sum_{i,j=0}^{d}\|\partial_{x_i, x_j} u\|_{L^2(\lambda\times \mu_U)}^2 &=\|\Lap u \|_{L^2(\lambda\times \mu_U)}^2-\int_{I\times\mathbb{R}^d} \nabla_x u^\top  \nabla_x^2 U\nabla_x u \ud t\ud\mu_U(x) \\ & =\|f^\perp \|_{L^2(\lambda\times \mu_U)}^2-\sum_{k=1}^d\int_{I\times\mathbb{R}^d} (\partial_{x_k} u)^2  U_k''(x_k) \ud t\ud\mu_U(x),
   \end{align*}
   this yields \eqref{eqn:diag} since $\phi_k=\partial_{x_k} u$.
		
   For the case $\Pi_v f=e^{\alpha t} w_\alpha$ for a particular
   $\alpha$, $\phi_k=\psi(t)\partial_{x_k}w_\alpha(x)$, where
   \[\|\psi(t)\|_{L^2(I)} \le \dfrac{1}{\alpha^2}\|\Pi_v
     f\|_{L^2(\lambda\times \mu_U)}.\] Moreover, again by Bochner's formula,
   using
   $\norm {\nabla_x^*\nabla_x w_\alpha}_{L^2(\mu_U)} = \alpha^2
   \norm{w_\alpha}_{L^2(\mu_U)}=\alpha^2$,
   \begin{align*}
     \sum_{i,j=1}^{d}\|\partial_{x_i, x_j} w_\alpha\|_{L^2(I\times
       \mu_U)}^2 & =\|\nabla_x^*\nabla_x w_\alpha \|_{L^2(\lambda\times \mu_U)}^2-\int_{I\times\mathbb{R}^d}
    \nabla_x w_\alpha^\top \nabla_x^2 U\nabla_x w_\alpha \ud\mu_U(x) \\ &=\alpha^4-\sum_{k=1}^d\int_{I\times\mathbb{R}^d}
     (\partial_{x_k} w_\alpha )^2 U_k''(x_k) \ud\mu_U(x).
   \end{align*}
   Therefore
   $\sum_{k=1}^d\int_{I\times\mathbb{R}^d} (\partial_{x_k} w_\alpha
   )^2 U_k''(x_k) \ud\mu_U(x) \leq \alpha^4$ and hence
   \begin{align*}
     \sum_{k=1}^d \int_{I\times\RR^d}   \phi_k^2\partial_{x_k x_k} U\ud t\ud \mu_U(x) & = \|\psi(t)\|_{L^2(I)}^2\sum_{k=1}^d\int_{I\times\mathbb{R}^d} (\partial_{x_k} w_\alpha )^2 U_k''(x_k) \ud\mu_U(x) \\ &\le \|\Pi_v f\|_{L^2(\lambda\times \mu_U)}^2.
   \end{align*}
   The estimate \eqref{eqn:diag} follows from linear
   combination. Substituting into \eqref{eqn:zzlem} we obtain \eqref{eqn:zigzag} with
   $R_{ZZ}=R$, so that we have a dimension-independent convergence rate assuming $U_k''(x_k)\ge -L$,
   even without an upper bound on $\nabla_x^2 U$ besides Assumption \ref{assump:hessian}. Moreover, if we further assume $U_k''(x_k)\ge 0$ for all $k$, then we have convergence rate $O(\sqrt{m})$ after optimizing in $\gamma$.
\end{remark}

\subsection{Bouncy particle sampler}
In this case $K=1$, and $n_1=\frac{\nabla_x U}{|\nabla_x U|}$. In order to avoid notation conflicts, in this section, we write $\nn = n_1$ and use 
\begin{equation*}
  \nn_i =\dfrac{\partial_{x_i} U}{|\nabla_x U|}
\end{equation*}
to denote the $i$-th component of $\nn$. As $\nn$ is normalized, $\sum_{i=1}^d \nn_i^2=1$. 

Recall that we want to estimate $\J$, which for BPS is given by 
\[\J=-\partial_t \phi_0+v\cdot\nabla_x\phi_0+v\cdot\partial_t
  \phi-\sum_{i=1}^d v_iv\cdot\partial_{x_i}\phi -2(-v\cdot \nn)_+(v\cdot
  \nn)(\phi\cdot \nabla_x U). \]
\begin{lemma}\label{lem:bps}
  Let $\phi_i$, $i=0,\cdots,d$ be the test functions as in Lemma
  \ref{lem:truetestfn}. Then
  \begin{equation}\label{eqn:bps}
    \int_{I\times \mathbb{R}^d} (\phi\cdot\nabla_x U)^2 \ud t\ud\mu_U(x) \le Cd \bigl(1+\dfrac{1}{\sqrt{m}T}+\dfrac{R}{\sqrt{m}}+RT\bigr)^2 \int_{I\times \mathbb{R}^d} (\Pi_v f)^2 \ud t\ud\mu_U(x).
  \end{equation}
  Here $R$ is defined as in Theorem \ref{thm:expconvfull}.
\end{lemma}
{Let us remark that the factor $d$ on the right-hand side of the above estimate is the reason that the convergence rate we obtain is degraded by a factor of $\sqrt{d}$ for BPS.
\begin{proof}
  By construction of the test functions \eqref{eqn:divergence}, we have
  \[\phi\cdot \nabla_x U =\Pi_v f+\partial_t \phi_0 + \sum_{i=1}^d
    \partial_{x_i}\phi_i.\]
  Thus 
  \begin{align*}
    \int_{I\times \mathbb{R}^d} (\phi\cdot\nabla_x U)^2
    \ud t\ud\mu_U(x)
    & = \int_{I\times \mathbb{R}^d} (\Pi_v
    f+\partial_t \phi_0 + \sum_{i=1}^d \partial_{x_i}\phi_i)^2 \ud
      t\ud\mu_U(x) \\
    & \le (d+2) \int_{I\times \mathbb{R}^d}
    \bigl((\Pi_v f)^2+(\partial_t \phi_0 )^2+ \sum_{i=1}^d
      (\partial_{x_i}\phi_i)^2\bigr) \ud t\ud\mu_U(x) \\
    &  \leftstackrel{\eqref{eqn:estdiv}}{\le}
    Cd(1+\dfrac{1}{\sqrt{m}T}+\dfrac{R}{\sqrt{m}}+RT)^2 \int_{I\times
      \mathbb{R}^d} (\Pi_v f)^2 \ud t\ud\mu_U(x),
  \end{align*}
  where the first inequality follows from Cauchy-Schwarz. 
\end{proof}

\begin{proof}[Proof of Lemma \ref{lem:cj} for bouncy particle sampler]

Similar to the proof for ZZ, we will expand $\| \J \|_{L^2(\lambda\times \stationary)}^2$ and organize its terms according to their powers on $v$ and whether $(-v\cdot \nn)_+$ appears in the expression. Terms that do not contain $(-v\cdot \nn)_+$ are identical to those for ZZ and thus calculations are omitted. 

Next we look at terms where $(-v\cdot \nn)_+$ appears twice, in which the overall power of $v$ is even so Lemma \ref{lem:vinteg} can be applied: \begin{equation}\label{eqn:bpsvnplsq}\begin{aligned}
     \int_{I\times \mathbb{R}^d\times \mathbb{R}^d}&  4(-v\cdot\nn)_+^2(v\cdot\nn)^2 (\phi\cdot\nabla_x U)^2 \ud t\ud\stationary(x,v)  \\ & =  \int_{I\times \mathbb{R}^d\times \mathbb{R}^d}  2(v\cdot\nn)^4 (\phi\cdot\nabla_x U)^2 \ud t\ud\stationary(x,v) \\ & =  \int_{I\times \mathbb{R}^d\times \mathbb{R}^d}  2\sum_{i,j,k,p=1}^d v_iv_jv_kv_p \nn_i \nn_j \nn_k\nn_p (\phi\cdot\nabla_x U)^2 \ud t\ud\stationary(x,v) \\ & =  \int_{I\times \mathbb{R}^d\times \mathbb{R}^d} (2\sum_{i=1}^dv_i^4 \nn_i^4+6\sum_{1\le i\neq j\le d}v_i^2v_j^2\nn_i^2\nn_j^2) (\phi\cdot\nabla_x U)^2 \ud t\ud\stationary(x,v) \\ & =  \int_{I\times \mathbb{R}^d} (6\sum_{i=1}^d\nn_i^4+6\sum_{1\le i\neq j\le d}\nn_i^2\nn_j^2) (\phi\cdot\nabla_x U)^2 \ud t\ud\mu_U(x) \\ & =  6\int_{I\times \mathbb{R}^d} (\sum_{i}\nn_i^2)^2 (\phi\cdot\nabla_x U)^2 \ud t\ud\mu_U(x) \\ & =  6\int_{I\times \mathbb{R}^d} (\phi\cdot\nabla_x U)^2 \ud t\ud\mu_U(x).
\end{aligned}\end{equation}
Cross terms with $(-v\cdot\nn)_+$ appearing once and the overall power of $v$ is even: \begin{align*}
    \int_{I\times \mathbb{R}^d\times \mathbb{R}^d}& \Bigl(4\partial_t\phi_0(-v\cdot\nn)_+(v\cdot\nn)(\phi\cdot \nabla_x U)+4\sum_{i,j=1}^d v_iv_j \partial_{x_i} \phi_j(-v\cdot\nn)_+(v\cdot\nn)(\phi\cdot \nabla_x U)\Bigr) \ud t\ud\stationary \\ & =-2\int_{I\times \mathbb{R}^d\times \mathbb{R}^d} \Bigl(\partial_t\phi_0(v\cdot\nn)^2+\sum_{i,j=1}^d v_iv_j \partial_{x_i} \phi_j(v\cdot\nn)^2\Bigr)(\phi\cdot \nabla_x U) \ud t\ud\stationary \\ & =-2\int_{I\times \mathbb{R}^d\times \mathbb{R}^d} \Bigl(\partial_t\phi_0\sum_{i=1}^d\nn_i^2+\sum_{i,j,p,q=1}^d v_iv_jv_pv_q\nn_p\nn_q \partial_{x_i} \phi_j\Bigr)(\phi\cdot \nabla_x U) \ud t\ud\stationary \\ & =-2\int_{I\times \mathbb{R}^d\times \mathbb{R}^d} \Bigl(\partial_t\phi_0+\sum_{i=1}^d v_i^4\nn_i^2 \partial_{x_i} \phi_i +\sum_{1\le i\neq j\le d} v_i^2 v_j^2\nn_j^2\partial_{x_i}\phi_i \\ & \hspace{10em} +2\sum_{1\le i\neq j\le d}v_i^2v_j^2 \nn_i\nn_j\partial_{x_i}\phi_j\Bigr)(\phi\cdot \nabla_x U) \ud t\ud\stationary \\ & =-2\int_{I\times \mathbb{R}^d} \Bigl(\partial_t\phi_0+3\sum_{i=1}^d \nn_i^2 \partial_{x_i} \phi_i +\sum_{1\le i\neq j\le d} \nn_j^2\partial_{x_i}\phi_i \\ & \hspace{10em} +2\sum_{1\le i\neq j\le d} \nn_i\nn_j\partial_{x_i}\phi_j\Bigr)(\phi\cdot \nabla_x U) \ud t\ud\mu_U(x) \\ & =-2\int_{I\times \mathbb{R}^d} \Bigl(\partial_t\phi_0+\sum_{i=1}^d \partial_{x_i} \phi_i   +2\sum_{i,j=1}^d \nn_i\nn_j\partial_{x_i}\phi_j\Bigr)(\phi\cdot \nabla_x U) \ud t\ud\mu_U(x)\\ & \le 2\int_{I\times \mathbb{R}^d} \Bigl(-(\partial_t\phi_0+\sum_{i=1}^d \partial_{x_i} \phi_i)(\phi\cdot \nabla_x U) + \sum_{i,j=1}^d(\nn_i^2\nn_j^2 (\phi\cdot\nabla_x U)^2 +(\partial_{x_i}\phi_j)^2)\Bigr)\ud t\ud\mu_U(x) \\ & = 2\int_{I\times \mathbb{R}^d} \Bigl(-(\partial_t\phi_0+\sum_{i=1}^d \partial_{x_i} \phi_i)(\phi\cdot \nabla_x U) +(\phi\cdot\nabla_x U)^2+ \sum_{i,j=1}^d(\partial_{x_i}\phi_j)^2\Bigr)\ud t\ud\mu_U(x) .
\end{align*}
Finally the cross terms with $(-v\cdot\nn)_+$ appearing once and an odd overall power on $v$, in which we again control by terms we have calculated above \begin{align*}
    -4\int_{I\times \mathbb{R}^d\times \mathbb{R}^d} & v\cdot (\partial_t \phi+\nabla_x\phi_0) (-v\cdot\nn)_+(v\cdot\nn)(\phi\cdot \nabla_x U) \ud t\ud\stationary\\ & \le 2 \int_{I\times \mathbb{R}^d\times \mathbb{R}^d}  \Bigl(|v\cdot (\partial_t \phi+\nabla_x\phi_0)|^2+ (-v\cdot\nn)_+^2(v\cdot\nn)^2(\phi\cdot \nabla_x U)^2 \Bigr)\ud t \ud \stationary \\ & \leftstackrel{\eqref{eqn:bpsvnplsq}}{\le}  \int_{I\times \mathbb{R}^d}  \Bigl(4|\partial_t \phi|^2+4|\nabla_x\phi_0|^2+ 3(\phi\cdot \nabla_x U)^2 \Bigr)\ud t \ud \mu_U(x).
\end{align*}
Therefore, combining these calculations, we obtain 
\begin{align*}
    \| \J  \|_{L^2(\lambda\times \stationary)}^2  
          & \le \int_{I\times \mathbb{R}^d} \Big((\partial_t\phi_0+\sum_{i=1}^d \partial_{x_i} \phi_i-\phi\cdot \nabla_x U)^2 +6\sum_{i,j=0}^d |\partial_{x_i} \phi_j|^2 \\ & \qquad +16(\phi\cdot\nabla_x U)^2   \Big)\ud t\ud\mu_U(x) \\
          & \leftstackrel{\eqref{eqn:divergence},\eqref{eqn:estdiv},\eqref{eqn:bps}}{\le} Cd(1+\dfrac{1}{\sqrt{m}T}+\dfrac{R}{\sqrt{m}}+RT)^2 \|\Pi_v f\|_{L^2(\lambda\times \mu_U)}^2. \qedhere
\end{align*}
\end{proof}
  
\section*{Acknowledgment}
This work is supported in part by National Science Foundation via grants CCF-1910571 and DMS-2012286.

\bibliographystyle{amsplain}
\bibliography{main}
	
\end{document}